\documentclass[12pt,reqno]{amsart}

\usepackage{amsmath,amsthm,amssymb,comment,fullpage}

\usepackage{times}
\usepackage[T1]{fontenc}
\usepackage{mathrsfs}
\usepackage{latexsym}
\usepackage{mathtools}
\usepackage{braket}
\usepackage[dvips]{graphics}
\usepackage{epsfig}
\usepackage{amsmath,amsfonts,amsthm,amssymb,amscd}
\input amssym.def
\input amssym.tex
\usepackage{color}
\usepackage{hyperref}
\usepackage{url}

\newcommand{\burl}[1]{\textcolor{blue}{\url{#1}}}

\newcommand{\lr}[1]{\left\lfloor#1\right\rfloor}

\numberwithin{equation}{section}

\newtheorem{thm}{Theorem}[section]

\newtheorem{cor}[thm]{Corollary}
\newtheorem{lem}[thm]{Lemma}
\newtheorem{prop}[thm]{Proposition}

\theoremstyle{plain}

\newtheorem{proposition}[thm]{Proposition}
\newtheorem{theorem}[thm]{Theorem}

\newcommand\be{\begin{equation}}
\newcommand\ee{\end{equation}}
\newcommand\bea{\begin{eqnarray}}
\newcommand\eea{\end{eqnarray}}
\newcommand\bi{\begin{itemize}}
\newcommand\ei{\end{itemize}}
\newcommand\ben{\begin{enumerate}}
\newcommand\een{\end{enumerate}}
\newcommand\bc{\begin{center}}
\newcommand\ec{\end{center}}
\newcommand\ba{\begin{array}}
\newcommand\ea{\end{array}}

\def\cn{\frac{1}{(\mu_n-1)a_{2n+1}}}
\def\cm{c_n}

\newcommand\frakfamily{\usefont{U}{yfrak}{m}{n}}
\DeclareTextFontCommand{\textfrak}{\frakfamily}

\newcommand{\ncr}[2]{{#1 \choose #2}}
\newcommand{\twocase}[5]{#1 \begin{cases} #2 & \text{{\rm #3}}\\ #4
&\text{{\rm #5}} \end{cases}   }

\newcommand{\hr}[1]{\href{#1}{\url{#1}}}

\newcommand{\sqe}{{\sqrt{1+8e^{r_n}}}}

\title{Generalizing Zeckendorf's Theorem: The Kentucky Sequence}

\author{Minerva Catral}
\email{\textcolor{blue}{\href{mailto:catralm@xavier.edu}{catralm@xavier.edu}}}
\address{Department of Mathematics and Computer Science, Xavier University, Cincinnati, OH 45207}

\author{Pari Ford}
\email{\textcolor{blue}{\href{mailto:fordpl@unk.edu}{fordpl@unk.edu}}}
\address{Department of Mathematics and Statistics, University of Nebraska at Kearney, Kearney, NE 68849}

\author{Pamela Harris}
\email{\textcolor{blue}{\href{mailto:Pamela.Harris@usma.edu}{Pamela.Harris@usma.edu}}}
\address{Department of Mathematical Sciences, United States Military Academy, West Point, NY 10996}

\author{Steven J. Miller}
\email{\textcolor{blue}{\href{mailto:sjm1@williams.edu}{sjm1@williams.edu}},  \textcolor{blue}{\href{Steven.Miller.MC.96@aya.yale.edu}{Steven.Miller.MC.96@aya.yale.edu}}}
\address{Department of Mathematics and Statistics, Williams College, Williamstown, MA 01267}

\author{Dawn Nelson}
\email{\textcolor{blue}{\href{mailto:dnelson1@saintpeters.edu}{dnelson1@saintpeters.edu}}}
\address{Department of Mathematics, Saint Peter's University, Jersey City, NJ 07306}

\thanks{The fourth named author was partially supported by NSF grant DMS1265673. This research was performed while the third named author held a National Research Council Research Associateship Award at USMA/ARL. This work was begun at the 2014 REUF Meeting at AIM; it is a pleasure to thank them for their support, and the participants there and at the 16\textsuperscript{th} International Conference on Fibonacci Numbers and their Applications for helpful discussions.}



\date{\today}

\begin{document}

\maketitle

\begin{abstract} By Zeckendorf's theorem, an equivalent definition of the Fibonacci sequence (appropriately normalized) is that it is the unique sequence of increasing integers such that every positive number can be written uniquely as a sum of non-adjacent elements; this is called a legal decomposition. Previous work examined the distribution of the number of summands and the spacings between them, in legal decompositions arising from the Fibonacci numbers and other linear recurrence relations with non-negative integral coefficients. Many of these results were restricted to the case where the first term in the defining recurrence was positive. We study a generalization of the Fibonacci numbers with a simple notion of legality which leads to a recurrence where the first term vanishes. We again have unique legal decompositions, Gaussian behavior in the number of summands, and geometric decay in the distribution of gaps.
\end{abstract}



\section{Introduction}


One of the standard definitions of the Fibonacci numbers $\{F_n\}$ is that it is     the unique sequence satisfying the recurrence $F_{n+1} = F_n + F_{n-1}$ with initial conditions $F_1 = 1$, $F_2 = 2$. An interesting and equivalent definition is that it is the unique increasing sequence of positive integers such that every positive number can be written uniquely as a sum of non-adjacent elements of the sequence.\footnote{If we started the Fibonacci numbers with a zero, or with two ones, we would lose uniqueness.} This equivalence is known as Zeckendorf's theorem \cite{Ze}, and frequently one says every number has a unique legal decomposition as a sum of non-adjacent Fibonacci numbers.

In recent years there has been a lot of research exploring other notions of a legal decomposition,  seeing which sequences result, and studying  the properties of the resulting sequences and decompositions (see for example \cite{Al,Day,DDKMMV,DDKMV,DG,FGNPT,GT,GTNP,Ke,Len,MW2,Ste1,Ste2} among others).
Most of the previous work has been on sequences $\{G_n\}$ where the recurrence relation coefficients are non-negative integers, with the additional restriction being that the first and last terms are positive\footnote{Thus $G_{n+1} = c_1 G_n + \cdots + c_L G_{n-(L-1)}$ with $c_1 c_L > 0$ and $c_i \ge 0$.} (see for instance \cite{MW1}, who call these \textbf{Positive Linear Recurrence Sequences}).

Much is known about the properties of the summands in decompositions. The first result is Lekkerkerker's theorem \cite{Lek}, which says the number of summands needed in the decomposition of $m \in [F_n, F_{n+1})$ grows linearly with $n$. Later authors extended this to other recurrences and found that the distribution of the number of summands converges to a Gaussian. Recently the distribution of gaps between summands in decompositions was studied; the distribution of the longest gap between summands converges to the same distribution one sees when looking at the longest run of heads in tosses of a biased coin, while the probability of observing a gap of length $g$ converges to a geometric random variable for $g > L$ (and is computable for smaller $g$, with the result depending on the recurrence); good sources on these recent gap results are \cite{BBGILMT,B-AM,BILMT}.

Our goal is to extend these results to recurrences that could not be handled by existing techniques. To that end, we study a sequence arising from a notion of a legal decomposition whose recurrence has first term equal to zero.\footnote{Thus in $G_{n+1} = c_1 G_n + \cdots + C_L G_{n-(L-1)}$ we have $c_1 = 0$.} While this sequence does fit into the new framework of an $f$-decomposition introduced in \cite{DDKMMV}, their arguments only suffice to show that our decomposition rule leads to unique decompositions, and is sadly silent on the distribution of the number of summands and the gaps between them; we remedy this below by completely resolving these issues in Theorems \ref{thm:gaussian} and \ref{thm:gapstheorem}.

We now describe our object of study. We can view the decomposition rule of the Fibonacci numbers as saying our sequence is divided into bins of length 1, and (i) we can use at most one element from a bin at most one time, and (ii) we cannot choose elements from two adjacent bins. This suggests a natural extension where the bins now contain $b$ elements and any two summands of a decomposition cannot be members of the same bin or any of the $s$ bins immediately before or any of the $s$ bins immediately after. We call this the \textbf{$(s,b)$-Generacci sequence}, and thus the Fibonacci numbers are the $(1,1)$-Generacci sequence. In this paper we consider the next simplest case: $s=1, b=2$. While the ideas needed to analyze this case carry over to the more general case, it is useful to specialize so that the technical details do not needlessly clutter arguments. For ease of exposition, we decided to give this special sequence a name, and are calling it the \textbf{Kentucky-2} sequence after the homestate of one of our authors.\footnote{The Kentucky-1 sequence is equivalent to the Fibonacci sequence.}

The elements of the Kentucky-2 sequence are partitioned into bins of size 2, and thus the $k$\textsuperscript{th} bin is \be b_k \ := \ \{a_{2k-1}, a_{2k}\}.\ee For a positive integer $m$, a Kentucky-2 legal decomposition is \be m \ = \ a_{\ell_1} + a_{\ell_2}  + \cdots + a_{\ell_k}, \ \ \ \ell_1\ <\ \ell_2 \ < \ \cdots\ <\ \ell_k\ee and $\{a_{\ell_j}, a_{\ell_{j+1}}\} \not\subset b_i \cup b_{i-1}$ for any $i,j$ (i.e., we cannot decompose a number using more than one summand from the same bin or two summands from adjacent bins). The first few terms of the Kentucky-2 sequence are
\begin{equation} \underbracket{\ 1,\ 2\ }_{b_1}\ ,\ \underbracket{\    3,\  4 \ }_{b_2}\ ,\ \underbracket{\      5,\ 8 \ }_{b_3}\ ,\ \underbracket{\     11,\  16 \ }_{b_4}\ ,\ \underbracket{\     21,\  32 \ }_{b_5}\ ,\ \underbracket{\     43,\  64 \ }_{b_6}\ , \ \underbracket{\      85,\ 128 \ }_{b_7}\ ,\ \underbracket{\     171,\  256 \ }_{b_8}\ ,\ \ldots\end{equation}

We have a nice closed form expression for the elements of this sequence.

\begin{theorem}\label{thm:recurrencevalues} If $\{a_n\}$ is the Kentucky-2 sequence, then \begin{equation}
a_{n+1} \ = \  a_{n-1} + 2a_{n-3}, \ a_1 \ = \  1, \ a_2 \ = \  2, \ a_3 \ = \  3, \ a_4 \ = \ 4,
\end{equation} which implies
\begin{equation} a_{2n}\ = \ 2^n\ \ \ \ {\rm and} \ \ \ \ a_{2n-1}\ =\ \frac{1}{3}\left(2^{n+1}+(-1)^n\right).\end{equation}
\end{theorem}

This is not a Positive Linear Recurrence Sequence as the leading coefficient (that of $a_n$) is zero, and this sequence falls outside the scope of many of the previous techniques. We prove the following theorems concerning the Kentucky-2 Sequence.

\begin{theorem}[Uniqueness of Decompositions]\label{K2unique} Every positive integer can be written uniquely as a sum of distinct terms from the Kentucky-2 sequence where no two summands are in the same bin and no two summands belong to consecutive bins in the sequence. \end{theorem}

While the above follows immediately from the work of Demontigny, Do, Kulkarni, Miller, Moon and Varma \cite{DDKMMV} on $f$-decompositions (take $f(n) = 3$ if $n$ is even and $f(n) = 2$ otherwise), for completeness we give an elementary proof in Appendix \ref{sec:proofofK2unique}. We next generalize the results on Gaussian behavior for the summands to this case.

\begin{theorem}[Gaussian Behavior of Summands]\label{thm:gaussian}
Let the random variable $Y_n$ denote the number of summands in the (unique) Kentucky-2 decomposition of an integer picked at random from $[0, a_{2n+1})$ with uniform probability.\footnote{Using the methods of \cite{BDEMMTTW}, these results can be extended to hold almost surely for sufficiently large sub-interval of $[0, a_{2n+1})$.} Normalize $Y_n$ to $Y_n' = (Y_n - \mu_n)/\sigma_n$, where $\mu_n$ and $\sigma_n$ are the mean and variance of $Y_n$ respectively, which satisfy \bea \mu_n & \ = \ &  \frac{n}{3}+\frac{2}{9} + O\left(\frac{n}{2^n}\right) \nonumber\\
\sigma_n^2 & \ = \ &  \frac{2n}{27}+\frac{8}{81}+O\left(\frac{n^2}{2^n}\right). \eea Then $Y_n'$ converges in distribution to the standard normal distribution as $n \rightarrow \infty$.
\end{theorem}

Our final results concern the behavior of gaps between summands. For the legal decomposition \be m\ = \ a_{\ell_1} + a_{\ell_2}  + \cdots + a_{\ell_k} \ \ \ {\rm with} \ \ \ \ell_1\ <\ \ell_2 \ < \  \cdots \ < \  \ell_k \ee and $m \in [0,a_{2n+1})$, we define the  set of gaps as follows:
\begin{equation}\text{Gaps}_n(m) \ := \  \{\ell_2-\ell_1, \ell_3 - \ell_2, \dots, \ell_k - \ell_{k-1}\};\end{equation} notice we do not include the wait to the first summand (we could if we wish; one additional gap will not affect the limiting behavior). We can do the analysis two different ways, either averaging over all $m \in [0, a_{2n+1})$ or for each $m$. It is easier to average over all such $m$, and in fact this analysis is the first step towards understanding the behavior of the individual gap measure. In this paper we concentrate on just the average gap measure, though with additional work the techniques from \cite{BILMT} should be applicable and should yield similar results for the individual gaps and the distribution of the longest gap measure. We plan to return to these questions in a later paper where we consider the general $(s,b)$-Generacci sequence.

Thus in the theorem below we consider all the gaps between summands in Kentucky-2  legal decompositions of all $m \in [0, a_{2n+1})$. We let $P_n(g)$ be the fraction of all these gaps that are of length $g$; thus $m = a_1 + a_{11} + a_{15} + a_{22} + a_{26}$ contributes two gaps of length 4, one gap of length 7 and one gap of length 10.

\begin{theorem}[Average Gap Measure]\label{thm:gapstheorem}
For $P_n(g)$ as defined above (the probability of a gap of length $g$ among Kentucky-2 legal decompositions of $m\in[0,a_{2n+1})$), the limit $P(g) := \lim_{n\to\infty}P_n(g)$ exists, and
\be P(0)  \ = \ P(1) \ = \ P(2) \ = \ 0, \ \ \ P(3) \ = \ 1/8,\ee and for $g\geq 4$ we have
\begin{equation} \twocase{P(g)\ = \ }{2^{-j}}{if $g=2j$}{\frac{3}{4} \ 2^{-j}}{if $g=2j+1$.}\end{equation}
\end{theorem}

In \S\ref{sec:recurrencerelations} we derive the recurrence relation and explicit closed form expressions for the terms of the Kentucky-2 sequence, as well as a useful generating function for the number of summands in decompositions. We then prove Theorem \ref{thm:gaussian} on Gaussian behavior in \S\ref{sec:gaussianbehavior}, and Theorem \ref{thm:gapstheorem} on the distribution of the gaps in \S\ref{sec:averagegap}. We end with some concluding remarks and directions for future research in \S\ref{sec:conclusionandfuturework}.

\section{Recurrence Relations and Generating functions}\label{sec:recurrencerelations}

In the analysis  below we constantly use the fact that every positive integer has a unique Kentucky-2 legal decomposition; see \cite{DDKMMV} or Appendix \ref{sec:proofofK2unique} for proofs.

\subsection{Recurrence Relations}

\begin{proposition} For the Kentucky-2 sequence, $a_n = n$ for $1 \le n \le 5$ and
for any $n \geq 5$ we have $a_n = a_{n-2} + 2a_{n-4}$.  Further  for $n \geq 1$ we have \be\label{eq:expterm} a_{2n}\ =\ 2^n \ \ \ {\rm and} \ \ \ a_{2n-1}\ =\ \frac{1}{3}\left(2^{n+1}+(-1)^n\right).\ee  \end{proposition}

\begin{proof}
Any $a_{2n+1}$ and $a_{2n}$ in the Kentucky-2 sequence is listed because it is the smallest integer that cannot be legally decomposed using the members of $\{a_1, a_2, \dots, a_{2n}\}$  or  $\{a_1, a_2, \dots, a_{2n-1}\}$ respectively: \begin{equation} \underbracket{\ 1,\ 2\ }_{b_1}\ ,\ \underbracket{\    3,\  4 \ }_{b_2}\ ,\ \underbracket{\      5,\ 8 \ }_{b_3}\ ,\ \underbracket{\     11,\  16 \ }_{b_4}\ ,\ \underbracket{\     21,\  32 \ }_{b_5}\ ,\ \underbracket{\     43,\  64 \ }_{b_6}\ ,\ \cdots, \  \underbracket{\     a_{2n-3},\  a_{2n-2}\ }_{b_{n-1}} \ , \ \underbracket{\ a_{2n-1},\ a_{2n}\ }_{b_n}.\end{equation}


As $a_{2n}$ is the largest entry in the bin $b_n$, it is one more than the largest number we can legally write, and thus \begin{equation}a_{2n} \ = \  a_{2n-1} + a_{2(n-2)} + a_{2(n-4)} + \cdots + a_j+1\end{equation} where $a_j = a_2$ if $n$ is odd and $a_j = a_4$ if $n$ is even.  By construction of the sequence we have $a_{2(n-2)} + a_{2(n-4)} + \cdots + a_j+1 = a_{2(n-2)+1} = a_{2n-3}$.  Thus
\begin{equation}\label{evenRR}
a_{2n} \ = \  a_{2n-1} + a_{2n-3}.
\end{equation}

Similarly $a_{2n+1}$ is the smallest entry in bin $b_{n+1}$, so \begin{equation}a_{2n+1} \ = \  a_{2n} + a_{2(n-2)} + a_{2(n-4)} + \cdots + a_j + 1\end{equation} where $a_j = a_2$ if $n$ is odd and $a_j = a_4$ if $n$ is even.  Thus
\begin{equation}\label{oddRR}
a_{2n+1} \ = \   a_{2n} + a_{2n-3}.
\end{equation}
Substituting Equation \eqref{evenRR} into \eqref{oddRR} yields
\begin{equation}\label{finaloddRR}
a_{2n+1} \ = \  a_{2n-1} + 2 a_{2n-3},
\end{equation} and thus for $m \geq 5$ odd we have $a_m = a_{m-2} + 2\cdot a_{m-4}$.

Now  using Equation \eqref{finaloddRR} in \eqref{evenRR}, we have \begin{align}
a_{2n} \ = \  a_{2n-1} + a_{2n-3} \ = \  a_{2n-3}+2\cdot a_{2n-5} + a_{2n-3}\ = \  2(a_{2n-3} + a_{2n-5}).
\end{align}
Shifting the index in Equation \eqref{evenRR} gives\begin{equation}\label{even_double}
a_{2n}\ = \  2 \cdot a_{2n-2}.
\end{equation}
Since $a_2 = 2$ and $a_4 = 4$, together with Equation \eqref{even_double} we now have $a_{2n} = 2^n$ for all $n \geq 1$.  A few algebraic steps then confirm $a_m = a_{m-2} + 2\cdot a_{m-4}$ for $m \geq 6$ even. 

Finally, we prove that $a_{2n-1} = \frac{1}{3}(2^{n+1}+(-1)^n)$  for $n \geq 1$ by induction.  The basis case is immediate as $a_1 = 1$ and $\frac{1}{3}(2^{1+1}+(-1)^1) = \frac{1}{3}(4-1) = 1$. Assume for some $N \geq 1$, $a_{2N-1} = \frac{1}{3}(2^{N+1}+(-1)^N)$.  By Equation \eqref{finaloddRR}, we have \begin{align}
a_{2(N+1)-1} &\ = \  a_{2N+1}\nonumber\\
&\ = \  a_{2N-1}+2\cdot a_{2N-3}\nonumber\\
&\ = \  \frac{1}{3}(2^{N+1}+(-1)^N) + 2\cdot \frac{1}{3}(2^{N-1+1}+(-1)^{N-1})\nonumber\\
&\ = \  \frac{1}{3}(2^{N+1} + (-1)^N + 2^{N+1} + (-1)^{N-1} + (-1)^{N-1})\nonumber\\
&\ = \  \frac{1}{3}(2^{N+2}+(-1)^{N+1}),
\end{align} and thus for all $n \geq 1$ we have $a_{2n-1} = \frac{1}{3}(2^{n+1}+(-1)^n)$.
\end{proof}

\subsection{Counting Integers With Exactly $k$ Summands}

In \cite{KKMW}, Kolo$\breve{{\rm g}}$lu, Kopp, Miller and Wang introduced a very useful combinatorial perspective to attack Zeckendorf decomposition problems. While many previous authors attacked related problems through continued fractions or Markov chains, they instead partitioned the $m \in [F_n, F_{n+1})$ into sets based on the number of summands in their Zeckendorf decomposition. We employ a similar technique here, which when combined with identities about Fibonacci polynomials allows us to easily obtain Gaussian behavior.

Let $p_{n,k}$ denote the number of $m\in[0,a_{2n+1})$ whose legal decomposition contains exactly $k$ summands where $k \geq 0$. We have  $p_{n,0} = 1$ for $n \ge 0$, $p_{0,k} = 0$ for $k > 0$, $p_{1,1}=2$, and $p_{n,k} = 0$ if $k > \lfloor{\frac{n+1}{2}}\rfloor$.  Also, by definition, \begin{equation}\displaystyle\sum_{k=0}^{\lfloor\frac{n+1}{2}\rfloor}p_{n,k} \ = \  a_{2n+1},\end{equation} and we have the following recurrence.

\begin{proposition}\label{recur}
For $p_{n,k}$ as above, we have
\begin{equation}p_{n,k}\ = \ 2p_{n-2,k-1}+p_{n-1,k}\end{equation} for $n \geq 2$ and   $k \leq  \lfloor{\frac{n+1}{2}}\rfloor$.
\end{proposition}

\begin{proof}
We partition the Kentucky-2 legal decompositions of all $m \in [0,a_{2n+1})$ into two sets, those that have a summand from bin $b_n$ and those that do not.

If we have a legal decomposition $m = a_{\ell_1}+ a_{\ell_2} + \cdots + a_{\ell_k}$ with $a_{\ell_k} \in b_n$, then $a_{\ell_{k-1}} \leq a_{2(n-2)}$ and there are two choices for $a_{\ell_k}$.  The number of legal decompositions of the form $a_{\ell_1}+ a_{\ell_2} + \cdots + a_{\ell_{k-1}}$ with $a_{\ell_{k-1}} \leq a_{2(n-2)}$ is $p_{n-2,k-1}$ (note the answer is independent of which value $a_{\ell_k} \in b_n$ we have).  Thus the number of legal decompositions of $m$ containing exactly $k$ summands with largest summand in bin $b_n$ is $2\cdot p_{n-2,k-1}$.

If $m \in [0, a_{2n+1})$ does not have a summand from bin $b_n$ in its decomposition, then $m \in [0, a_{2n-1})$, and by definition the number of such $m$ with exactly $k$ summands in a legal decomposition is $p_{n-1,k}$.

Combining these two cases yields
\begin{equation}p_{n,k}\ = \ 2p_{n-2,k-1}+p_{n-1,k},\end{equation} completing the proof.\end{proof}

This recurrence relation allows us to compute a closed-form expression for $F(x,y)$, the generating function of the $p_{n,k}$'s.

\begin{proposition}
Let  \be F(x,y)\ :=\ \sum_{n,k\geq0}p_{n,k} x^ny^k\ee be the generating function of the $p_{n,k}$'s arising from Kentucky-2 legal decompositions. Then
\be\label{cform}
F(x,y) \ = \  \frac{1+2xy}{1-x-2x^2y}.  \ee
\end{proposition}

\begin{proof}
Noting that $p_{n,k} = 0$ if either $n<0$ or $k<0$, using  explicit values of $p_{n,k}$ and the recurrence relation from Proposition \ref{recur}, after some straightforward algebra we obtain
\begin{equation} F(x,y) \ = \  2x^2yF(x,y)+xF(x,y) +2xy+1. \end{equation}
From this, Equation \eqref{cform} follows.
\end{proof}

While the combinatorial vantage of \cite{KKMW} has been fruitfully applied to a variety of recurrences (see \cite{MW1,MW2}), their simple proof of Gaussianity does not generalize. The reason is that for the Fibonacci numbers (which are also the $(1,1)$-Generacci numbers) we have an explicit, closed form expression for the corresponding $p_{n,k}$'s, which greatly facilitate the analysis. Fortunately for us a similar closed form expression exists for Kentucky-2 decompositions.

\begin{proposition}\label{Kentucky-Pnk-explicit} Let $p_{n,k}$ be the number of integers in $[0,a_{2n+1})$ that have exactly $k$ summands in their Kentucky-2 legal decomposition. For all $k \geq 1$ and $n \geq 1 +2(k-1)$, we have \begin{equation}p_{n,k} \ = \  2^k {{n-(k-1)}\choose{k}}.\end{equation}
\end{proposition}

\begin{proof}
We are counting decompositions of the form $a^\prime_{\ell_1} + \cdots + a^\prime_{\ell_k}$ where $a^\prime_{\ell_i} \in b_{\ell_i} = \{a_{2\ell_i - 1}, a_{2\ell_i}\}$ and $\ell_i \leq n$. Define $x_1 := \ell_1-1$ and $x_{k+1} := n-\ell_{k}$.  For $2 \leq i \leq k$, define $x_i :=\ell_i - \ell_{i-1}-1$.  We have
\begin{equation}x_1 + 1 + x_2 + 1 + x_3 + 1 +\cdots + x_k + 1 + x_{k+1} \ = \  n.\end{equation}

We change variables to rewrite the above. Essentially what we are doing is replacing the $x$'s with new variables to reduce our Diophantine equation to a standard form that has been well-studied. As we have a legal decomposition, our bins must be separated by at least one and thus $x_i \ge 1$ for $2 \le i \le k-1$ and $x_1, x_k \ge 0$. We remove these known gaps in our new variables by setting $y_1:= x_1$, $y_{k+1}:=x_{k+1}$ and $y_i:=x_i-1$ for $2 \leq i \leq k$, which gives
\begin{align}\label{eq:yversioncookieproblem}
y_1 + y_2 + \cdots + y_{k} + y_{k+1} &\ = \  x_1 + (x_2-1) + \cdots + (x_k-1) + x_{k+1}\nonumber\\
&\ = \  n-k-(k-1).
\end{align} Finding the number of non-negative integral solutions to this Diophantine equation has many names (the Stars and Bars Problem, Waring's Problem, the Cookie Problem). As the number of solutions to $z_1 + \cdots + z_P = C$ is $\ncr{C+P-1}{P-1}$ (see for example \cite{MT-B,Na}, or \cite{MBD} for a proof and an application of this identity in Bayesian analysis), the number of solutions to Equation \eqref{eq:yversioncookieproblem} is given by the binomial coefficient \begin{equation}\ncr{n-k-(k-1)+k}{k} \ = \ \ncr{n-(k-1)}{k}. \end{equation}
As there are two choices for each $a^\prime_{\ell_i}$, we have $2^k$ legal decompositions whose summands are from the bins $\{b_{\ell_1}, b_{\ell_2}, \dots, b_{\ell_k}\}$ and thus \begin{equation}p_{n,k} \ = \  2^k \ncr{n-(k-1)}{k}.\end{equation}
\end{proof}

\section{Gaussian Behavior}\label{sec:gaussianbehavior}



Before launching into our proof of Theorem \ref{thm:gaussian}, we provide some numerical support in Figure \ref{fig:kentuckyplot}. We randomly chose 200,000 integers from $[0,10^{600})$. We observed a mean number of summands  of 666.899, which fits beautifully with the predicted value of 666.889; the standard deviation of our sample was 12.154, which is in excellent agreement with the prediction of 12.176.

\begin{center}
\begin{figure}
\includegraphics[scale=1]{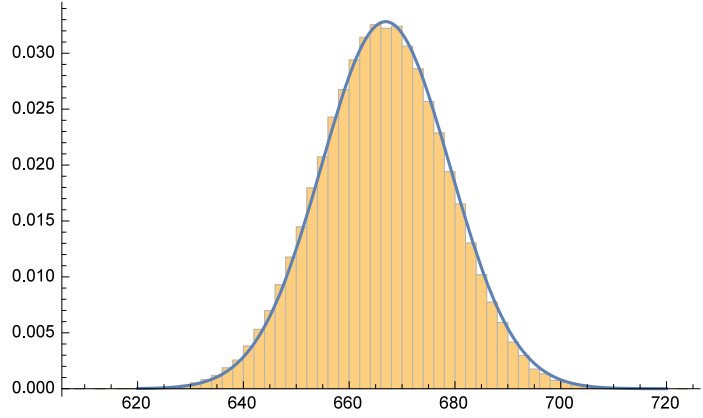}
\caption{\label{fig:kentuckyplot} The distribution of the number of summands in Kentucky-2  legal decompositions for 200,000 integers from $[0, 10^{600})$. }
\end{figure}
\end{center}

We split Theorem \ref{thm:gaussian} into three parts: a proof of our formula for the mean, a proof of our formula for the variance, and a proof of Gaussian behavior. We isolate the first two as separate propositions; we will prove these after first deriving some useful properties of the generating function of the $p_{n,k}$'s.


\begin{prop}
\label{prop:mean}
The mean number of summands in the Kentucky-2 legal decompositions for integers in $[0, a_{2n+1})$ is
\begin{equation}
  \mu_n\ = \ \frac{n}{3}+\frac{2}{9} + O\left(\frac{n}{2^n}\right).
   \end{equation}
\end{prop}

\begin{prop}
\label{prop:var}
The variance $\sigma_n^2$ of $Y_n$ (from Theorem \ref{thm:gaussian}) is
\begin{equation}
\sigma_n^2 \ = \   \frac{2n}{27}+\frac{8}{81}+O\left(\frac{n^2}{2^n}\right).
\end{equation}
\end{prop}

\subsection{Mean and Variance}\label{sec:mv}

Recall $Y_n$ is the random variable denoting the number of summands in the unique Kentucky-2 decomposition of an integer chosen uniformly from $[0, a_{2n+1})$, and $p_{n,k}$ denotes the number of integers in  $[0,a_{2n+1})$ whose legal decomposition contains exactly $k$ summands. The following lemma yields expressions for the mean and variance of $Y_n$ using a generating function for the $p_{n,k}$'s; in fact, it is this connection of derivatives of the generating function to moments that make the generating function approach so appealing. The proof is standard (see for example \cite{DDKMMV}).

\begin{lem}  \label{lem:mu}  \cite[Propositions 4.7, 4.8]{DDKMMV}  Let $F(x,y): =  \sum_{n, k \geq 0} p_{n,k} x^n y^k$ be the generating function of   $p_{n,k}$, and let  $g_n(y): =  \sum_{k=0}^n p_{n,k} y^k$ be the coefficient of $x^n$ in $F(x,y)$. Then the mean of $Y_n$ is
\begin{equation}
\mu_n \ = \  \frac{g_n'(1)}{g_n(1)},
\end{equation}
and the  variance of $Y_n$  is
\begin{equation}
\sigma_n^2 \ = \   \frac{\frac{d}{dy}(yg_n'(y))|_{y=1}}{g_n(1)} - \mu_n^2.
\end{equation}
\end{lem}

In our analysis  our closed form expression of $p_{n,k}$ as a binomial coefficient is crucial in obtaining simple closed form expressions for the needed quantities. We are able to express these needed quantities in terms of the \textbf{Fibonacci polynomials}, which are defined recursively as follows:
\begin{equation}
F_0(x) \ = \  0, \ F_1(x) \ = \  1, \ F_2(x) \ = \  x,
\end{equation}
and for $n \geq 3$
\begin{equation}
F_n(x) \ = \  xF_{n-1}(x) + F_{n-2}(x).
\end{equation}
For $n \geq 3$, the Fibonacci polynomial\footnote{Note that $F_n(1)$ gives the standard Fibonacci sequence.} $F_n(x)$  is given by
\begin{equation}
\label{eq:fibsum}
F_n(x) \ = \   \sum_{j=0}^{\lr{\frac{n-1}{2}}}  {n-j-1 \choose j} x^{n-2j-1},
\end{equation}
and also has the explicit formula
\begin{equation}
\label{eq:fibexp}
F_n(x) \ = \  \frac{(x+\sqrt{x^2+4})^n - (x-\sqrt{x^2+4})^n}{2^n \sqrt{x^2+4}}.
\end{equation}
The derivative of $F_n(x)$ is given by
\be
\label{eq:fibDeriv}
F_n'(x) \ = \ \frac{2nF_{n-1}(x) + (n-1)xF_n(x)}{x^2+4}.
\ee
For a reference on Fibonacci polynomials and the formulas given above (which follow immediately from the definitions and straightforward algebra), see \cite{Kos}.

\begin{prop}
For $n \geq 3$
\be
\label{eq:gnFib}
g_n(y)  \ = \  (\sqrt{2y})^{n+1} F_{n+2}\left(\frac{1}{\sqrt{2y}}\right).
\ee
\end{prop}

\begin{proof}
By Proposition \ref{Kentucky-Pnk-explicit}, we have
\begin{equation}
\ba{lcl}
F(x,y)  & \ = \  &  \displaystyle  \sum_{n=0}^\infty \sum_{k=0}^\infty p_{n,k} x^n y^k \ = \  \sum_{n=0}^\infty \sum_{k=0}^n 2^k {n-k+1 \choose k} x^n y^k.
\ea
\end{equation}
Thus, using Equation \eqref{eq:fibsum} we find
 \begin{align}
F(x,y) & \ = \   \displaystyle   \frac{1}{x^2\sqrt{2y}} \sum_{n=0}^\infty \sum_{k=0}^{n+2}   {(n+2)-k-1 \choose k} \left(\frac{1}{\sqrt{2y}}\right)^{(n+2)-2k-1}( x\sqrt{2y})^{n+2} \nonumber\\
&  \ = \     \displaystyle \frac{1}{x^2\sqrt{2y}} \sum_{n=0}^\infty F_{n+2}\left(\frac{1}{\sqrt{2y}}\right) ( x\sqrt{2y})^{n+2}  \ = \  \displaystyle  \sum_{n=0}^\infty F_{n+2}\left(\frac{1}{\sqrt{2y}}\right) (\sqrt{2y})^{n+1} x^n,
\end{align} completing the proof.
\end{proof}

In Appendix \ref{app:gfproof} we provide alternate proofs of Proposition \ref{prop:mean}, Proposition \ref{prop:var} and Theorem \ref{thm:gaussian} that follow directly from the recurrence for $p_{n,k}$ and properties of generating functions, as these arguments generalize better to other recurrences (this is similar to the difference in proofs in \cite{KKMW} and \cite{MW1}, where the first exploits the closed form expressions while the second argues more generally). In doing so, we find another formula for $g_n(y)$. This formula gives an independent derivation of the explicit formula for the Fibonacci polynomials, Equation \eqref{eq:fibexp}.

\begin{proof}[Proof of Proposition \ref{prop:mean}]
By Lemma \ref{lem:mu}, the mean of $Y_n$ is $g_n'(1)/g_n(1)$. Calculations of derivatives using Equations \eqref{eq:fibDeriv} and \eqref{eq:gnFib} give
\begin{align}
\displaystyle \frac{g_n'(1)}{g_n(1)}  & \ = \    \displaystyle \frac{(n+1)(\sqrt{2})^{n-1}F_{n+2}(\frac{1}{\sqrt{2}})}{F_{n+2}(\frac{1}{\sqrt{2}})(\sqrt{2})^{n+1}} - \frac{(\sqrt{2})^{n-2}F_{n+2}'(\frac{1}{\sqrt{2}})}{F_{n+2}(\frac{1}{\sqrt{2}})(\sqrt{2})^{n+1}} \nonumber\\
&\ = \  \frac{n+1}{2} - \frac{1}{(\sqrt{2})^3} \frac{F_{n+2}'\left(\frac{1}{\sqrt{2}}\right)}{F_{n+2}\left(\frac{1}{\sqrt{2}}\right)}. \nonumber\\
& \ = \   \displaystyle \frac{n+1}{2} - \frac{2(n+2)F_{n+1}\left(\frac{1}{\sqrt{2}}\right) + \frac{n+1}{\sqrt{2}}F_{n+2}\left(\frac{1}{\sqrt{2}}\right)}{9\sqrt{2}F_{n+2}\left(\frac{1}{\sqrt{2}}\right)}\nonumber\\ & \ = \  \frac{4}{9}(n+1)  - \frac{\sqrt{2}}{9} (n+2)  \frac{F_{n+1}\left(\frac{1}{\sqrt{2}}\right)}{F_{n+2}\left(\frac{1}{\sqrt{2}}\right)} \nonumber\\
& \ = \    \displaystyle  \frac{4}{9}(n+1)  - \frac{\sqrt{2}}{9} (n+2)  \left(\frac{1}{\sqrt{2}} + O(2^{-n}) \right)  \ = \  \frac{n}{3} + \frac{2}{9} + O(n2^{-n}),
\end{align}
where in the next to last step we use Equation \eqref{eq:fibexp} to approximate $F_{n+1}(1/\sqrt{2})/F_{n+2}(1/\sqrt{2})$.
\end{proof}

\begin{proof}[Proof of Proposition \ref{prop:var}]
By Lemma \ref{lem:mu},
\begin{equation}
\sigma_n^2 \ = \  \displaystyle \frac{g_n''(1)}{g_n(1)} + \frac{g_n'(1)}{g_n(1)} - \mu_n^2 \ = \  \frac{g_n''(1)}{g_n(1)}  + \mu_n(1 - \mu_n).
\end{equation}
Now,
\begin{equation}
\displaystyle \frac{g_n''(1)}{g_n(1)}  \ = \  \displaystyle \frac{(-2n+1)}{4\sqrt{2}} \frac{F_{n+2}'(\frac{1}{\sqrt{2}})}{F_{n+2}(\frac{1}{\sqrt{2}})} + \frac{(n^2-1)}{4} + \frac{1}{8} \frac{F_{n+2}''(\frac{1}{\sqrt{2}})}{F_{n+2}(\frac{1}{\sqrt{2}})}.
\end{equation}
Applying the derivative formula in Equation \eqref{eq:fibDeriv} and using \eqref{eq:fibexp}, we find
\begin{align}
\displaystyle \frac{F_{n+2}'(\frac{1}{\sqrt{2}})}{F_{n+2}(\frac{1}{\sqrt{2}})}   & \ = \   \displaystyle  \frac{4(n+2)}{9} \frac{F_{n+1}(\frac{1}{\sqrt{2}})}{F_{n+2}(\frac{1}{\sqrt{2}})}  + \frac{\sqrt{2}(n+1)}{9}\nonumber\\
& \ = \    \displaystyle  \frac{4(n+2)}{9} \left[\frac{1}{\sqrt{2}} + O(2^{-n})\right] + \frac{\sqrt{2}(n+1)}{9}
\end{align}

and
\begin{equation}
\begin{array}{lcl}
\displaystyle \frac{F_{n+2}''(\frac{1}{\sqrt{2}})}{F_{n+2}(\frac{1}{\sqrt{2}})}   & \ = \  &  \displaystyle  \frac{16(n^2+3n+2)}{81} \frac{F_{n}(\frac{1}{\sqrt{2}})}{F_{n+2}(\frac{1}{\sqrt{2}})}  + \frac{4\sqrt{2}(2n^2+3n-2)}{81} \frac{F_{n+1}(\frac{1}{\sqrt{2}})}{F_{n+2}(\frac{1}{\sqrt{2}})} +  \frac{2(n^2+9n+8)}{81} \\
 & \ = \  &  \displaystyle  \frac{16(n^2+3n+2)}{81} \left[\frac{1}{2} + O(2^{-n})\right] + \frac{4\sqrt{2}(2n^2+3n-2)}{81} \left[\frac{1}{\sqrt{2}} + O(2^{-n})\right] \\
 & & \displaystyle\ +\  \frac{2(n^2+9n+8)}{81}.
\end{array}
\end{equation}
Thus
\begin{equation}
\begin{array}{lcl}
\sigma_n^2  & \ = \  &  \displaystyle \frac{(-2n+1)}{4\sqrt{2}} \left[\frac{\sqrt{2}}{9}(3n+5)+ O(n2^{-n})\right]  +  \frac{(n^2-1)}{4} + \frac{1}{8} \left[\frac{2n^2}{9} + \frac{2n}{3} + \frac{8}{27} +  O(n^22^{-n})\right] \\
&  & \displaystyle + \left[ \frac{n}{3}+\frac{2}{9} + O\left(\frac{n}{2^n}\right)\right] \left[ 1 -  \frac{n}{3} - \frac{2}{9} + O\left(\frac{n}{2^n}\right)\right]  \ = \  \displaystyle \frac{2n}{27} + \frac{8}{81} + O\left(\frac{n^2}{2^n}\right),
\end{array}
\end{equation} completing the proof.
\end{proof}

\subsection{Gaussian Behavior}

\begin{proof}[Proof of Theorem \ref{thm:gaussian}]
 We prove that  $Y_n'$ converges in distribution to the standard normal distribution as $n \rightarrow \infty$ by  showing  that the moment generating function of $Y_n'$ converges to that of the standard normal (which is $e^{t^2/2}$).  Following the same argument as in \cite[Lemma 4.9]{DDKMMV}, the moment generating function $M_{Y_n'}(t)$ of $Y_n'$ is
\begin{equation}
M_{Y_n'}(t) \ = \  \displaystyle \frac{g_n(e^{t/\sigma_n})e^{-t\mu_n/\sigma_n}}{g_n(1)}.
\end{equation}
Thus we have
\begin{equation}
M_{Y_n'}(t) \ = \  \displaystyle \frac{F_{n+2} \left(\frac{1}{\sqrt{2e^{t/\sigma_n}}} \right) e^{(\frac{n+1}{2}- \mu_n)t/\sigma_n}}{F_{n+2}\left(\frac{1}{\sqrt{2}}\right)},
\end{equation}
and
\begin{equation}
\log(M_{Y_n'}(t)) \ = \  \log F_{n+2} \left(\frac{1}{\sqrt{2e^{t/\sigma_n}}} \right)  + \frac{t}{\sigma_n} \left(\frac{n+1}{2}- \mu_n \right) -  \log F_{n+2}\left(\frac{1}{\sqrt{2}}\right).
\end{equation}
From Equation \eqref{eq:fibexp},
\begin{equation}
\begin{array}{lcl}
F_{n+2}(x)  & \ = \  &   \displaystyle \frac{(x+\sqrt{x^2+4})^{n+2}}{2^{n+2} \sqrt{x^2+4}} \left[ 1 - \left(\frac{x-\sqrt{x^2+4}}{x+\sqrt{x^2+4}}\right)^{n+2}\right].\\
 \end{array}
\end{equation}
Thus
\begin{align}
\log F_{n+2}(x)  \ = & \  (n+2) \log(x+\sqrt{x^2+4}) - (n+2)\log 2 \nonumber\\
&\ \ \ \ - \frac{1}{2}\log(x^2+4)  +  \log  (1 - r(x)^{n+2})\nonumber\\
  = &  \ (n+2) \log x + (n+2) \log\left(1 + \sqrt{1 + \frac{4}{x^2}}\right) - (n+2)\log2  \nonumber\\
 & \ \ \ \   - \frac{1}{2}\log(x^2+4)  \ +  \  O(r(x)^{n}),\
\end{align}
where for all $x$
\begin{equation}
r(x) \ = \   \left(\frac{x-\sqrt{x^2+4}}{x+\sqrt{x^2+4}}\right) \ \in \  (0,1].
\end{equation}
Thus
\begin{equation}
\begin{array}{lcl}
\log F_{n+2}(\frac{1}{\sqrt{2}})  & \ = \  &   \frac{1}{2}(n+3)\log 2 - \log 3 + O(2^{-n}) \\
\end{array}
\end{equation}
and
\begin{align}
\log F_{n+2}\left(\frac{1}{\sqrt{2e^{t/ \sigma_n}}}\right)   \ = \  & \displaystyle -\frac{(n+2)}{2}\log 2 - \frac{(n+2)}{2\sigma_n}t  - (n+2)\log 2 \nonumber\\ & \ \ +\  (n+2)\alpha_n(t) - \frac{1}{2}\beta_n(t) +  O(r^n),
\end{align}
where
\begin{equation}
\alpha_n(t) \ = \ \log(1 + \sqrt{1+8e^{t/\sigma_n}}), \ \
\beta_n(t) \ = \  \log\left (\frac{1}{2}e^{-t/\sigma_n} + 4 \right),
\end{equation}
and
\begin{equation}
r \ = \  r\left(\frac{1}{\sqrt{2e^{t/ \sigma_n}}}\right)\ <\ 1.
\end{equation}
The Taylor series expansions for $\alpha_n(t)$ and $\beta_n(t)$ about $t=0$  are given by
\begin{equation}
\alpha_n(t) \ = \  \log 4 + \frac{1}{3\sigma_n}t + \frac{1}{27\sigma_n^2}t^2 + O(n^{-3/2})
\end{equation}
and
 \begin{equation}
\beta_n(t) \ = \  \log\left(\frac{9}{2}\right)  - \frac{1}{9\sigma_n}t + \frac{4}{81\sigma_n^2}t^2 + O(n^{-3/2}).
\end{equation}
Going back to $\log(M_{Y_n'}(t))$ we now have
 \begin{equation}
 \begin{array}{lcl}
 \log(M_{Y_n'}(t)) &  \ = \  &  \displaystyle  -\frac{3}{2}(n+2)\log 2 - \frac{(n+2)}{2\sigma_n}t  +  (n+2)\left[2\log 2 + \frac{1}{3\sigma_n}t + \frac{1}{27\sigma_n^2}t^2 + O(n^{-3/2})\right] \\
& & \displaystyle  - \frac{1}{2} \left[2\log 3  - \log 2 + O(n^{-1/2})\right]  +  \frac{(n+1 - 2 \mu_n)}{2\sigma_n} t -  \frac{1}{2}(n+3)\log 2  + \log 3 \\
& & + O(2^{-n}) + O(r^{n})  \\
 & \ = \  &  \displaystyle  -\frac{(2\mu_n+1)}{2\sigma_n}t + \frac{(n+2)}{3\sigma_n}t  + \frac{(n+2)}{27\sigma_n^2}t^2 + O(n^{-1/2})    + O(2^{-n}) + O(r^{n}). \\
 \end{array}
 \end{equation}
Since $\mu_n \sim \frac{n}{3}$ and $\sigma_n^2 \sim \frac{2n}{27}$,  it follows that  $\log(M_{Y_n'}(t)) \rightarrow \frac{1}{2}t^2$ as $n \rightarrow \infty$. As this is the moment generating function of the standard normal, our proof is completed.
\end{proof}

\section{Average Gap Distribution}\label{sec:averagegap}


In this section we prove our results about the behavior of gaps between summands in Kentucky-2 decompositions. The advantage of studying the average gap distribution is that, following the methods  of  \cite{BBGILMT,BILMT},  we reduce the problem to a combinatorial one involving how many $m \in [0, a_{2n+1})$ have a gap of length $g$ starting at a given index $i$. We then write the gap probability as a double sum over integers $m$ and starting indices $i$, interchange the order of summation, and invoke our combinatorial results.

While the calculations are straightforward once we adopt this perspective, they are long. Additionally, it helps to break the analysis into different cases depending on the parity of $i$ and $g$, which we do first below and then use those results to determine the probabilities.


\begin{proof}[Proof of Theorem \ref{thm:gapstheorem}]
Let $I_{n}:=[0,a_{2n+1})$ and let $m\in I_{n}$ with the legal decomposition \begin{equation}m\ = \ a_{\ell_1} + a_{\ell_2}  + \cdots + a_{\ell_k}, \end{equation} with $\ell_1 < \ell_2 < \cdots < \ell_k$. For $1\leq i, g\leq n$ we define $X_{i,g}(m)$ as an indicator function which denotes whether the decomposition of $m$ has a gap of length $g$ beginning at $i$. Formally,
\begin{equation}X_{i,g}(m)\ = \ \begin{cases}1&\mbox{if $\exists \ j, \  1 \leq j \leq k$ with $i = \ell_j$ and $i+g = \ell_{j+1}$}\\0&\mbox{otherwise}.\end{cases}\end{equation}
Notice when $X_{i,g}(m)=1$, this implies that there exists a gap between $a_i$ and $a_{i+g}$. Namely $m$ does not contain $a_{i+1},\ldots,a_{i+g-1}$ as summands in its legal decomposition.

As the definition of the Kentucky-2 sequence implies $P(g)=0$ for  $0\leq g\leq 2$, we assume below that $g \ge 3$. Hence if $a_j$ is a summand in the legal decomposition of $m$ and  $a_j<a_i$, then the admissible $j$ are at most $i-4$ if and only if $i$ is even, whereas the admissible $j$ are at most $i-3$ if and only if $i$ is odd. We are interested in computing the fraction of gaps (arising from the decompositions of all $m\in I_n$) of length $g$. This probability is given by
\begin{equation}P_n(g)\ = \ c_n \displaystyle\sum_{m=0}^{a_{2n+1}-1}\displaystyle\sum_{i=1}^{2n-g}X_{i,g}(m),\end{equation}
where \be\label{eq:whatiscn} c_n\ =\ \cn.\ee

To compute the above mentioned probability we argue based on the parity of $i$. We find the contribution of gaps of length $g$ from even $i$ and odd $i$ separately and then add these two. The case when $g=3$ is a little simpler, as only even $i$ contribute (if $i$ were odd and $g=3$ we would violate the notion of Kentucky-2  legal). \\ \

\noindent\textbf{Part 1 of the Proof: Gap Preliminaries:} \\ \

\noindent \textbf{Case 1: $i$ is even:} Suppose that $i$ is even.  This means that $a_i$ is the largest entry in its bin. Thus the largest possible summand less than $a_i$ would be $a_{i-4}$. First we need to know the number of legal decompositions that only contain summands from $\{a_1,\ldots,a_{i-4}\}$, but this equals the number of integers that lie in $[0,a_{2\left(\frac{i-4}{2}\right)+1}) \ = [0,a_{i-3})$. By Equation \eqref{eq:expterm}, this is given by
\begin{equation} \ a_{2\left(\frac{i-4}{2}\right)+1}{\ = a_{i-3}} \ = \  \frac{1}{3}(2^{\frac{i}{2}}+(-1)^{\frac{i-2}{2}}).\end{equation}

Next we must consider the possible summands between $a_{i+g}$ and $a_{2n+1}$. There are two cases to consider depending on the parity of $i+g$. \\ \

\textbf{Subcase (i): $g$ is even:} Notice that if $i+g$ is even (that is when $g$ is even) and $a_j$ is a summand in the legal decomposition of $m$ with $a_{i+g} < a_j$, then $j \geq {i+g+3}$. In this case the number of  legal decompositions only containing summands from the set $\{a_{i+g+3}, a_{i+g+4}, \dots, a_{2n}\}$ is the same as the number of integers that lie in $[0, a_{(2n-(i+g+2)) + 1})$, which equals \be a_{(2n-(i+g+2)) + 1}\  = \ a_{2\left(\frac{2n-(i+g+2)}{2}+1\right) - 1}\ =\ \frac{1}{3}\left(2^{\frac{2n-(i+g)}{2}+1}+(-1)^{\frac{2n-(i+g)}{2}}\right).\ee
So for a fixed even $i,g$, the number of $m\in I_n$ that have a gap of length $g$ beginning at $i$ is
\be\frac{1}{9}(2^{\frac{i}{2}}+(-1)^{\frac{i-2}{2}})(2^{\frac{2n-(i+g)}{2}+1}+(-1)^{\frac{2n-(i+g)}{2}}).
\ee
Hence in this case  we have that
\begin{align}
\sum_{m=0}^{a_{2n+1}-1} \sum_{i =1 \atop i, g\ {\rm even}}^{2n-g}X_{i,g}(m)&\ = \  
\frac{1}{9} \sum_{i =1 \atop i, g\ {\rm even}}^{2n-g}(2^{\frac{i}{2}}+(-1)^{\frac{i-2}{2}})(2^{\frac{2n-(i+g)}{2}+1}+(-1)^{\frac{2n-(i+g)}{2}}).
\end{align}
\ \\

\textbf{Subcase (ii): $g$ is odd:}  In the case when $i$ is even and $g$ is odd, {any legal decomposition of an integer $m\in I_n$ with a gap from $i$ to $i+g$ that contains summands $  a_j>a_{i+g}$} must have $j \geq i+g+4$. The number of  legal decompositions achievable only with summands in the set $\{a_{i+g+4}, a_{i+g+5}, \dots, a_{2n}\}$ is the same as the number of integers in the interval $[0, a_{2n-(i+g+2)})$, which is given by \be a_{2n-(i+g+2)} \ = \ a_{2\left(\frac{2n-(i+g+1)}{2}\right)-1}\ = \ \frac{1}{3}\left(2^{\frac{2n-(i+g+1)}{2}+1}+(-1)^{\frac{2n-(i+g+1)}{2}}\right).\ee Hence when $i$ is even and $g$ is odd we have that
\begin{align}
\sum_{m=0}^{a_{2n+1}-1} \sum_{i =1 \atop i\ {\rm even}, g\ {\rm odd}}^{2n-g} X_{i,g}(m) 
&\ = \ \frac{1}{9} \sum_{i =1\atop i\ {\rm even}, g\ {\rm odd}}^{2n-g} (2^{\frac{i}{2}}+(-1)^{\frac{i-2}{2}}) \left(2^{\frac{2n-(i+g+1)}{2}+1}+(-1)^{\frac{2n-(i+g+1)}{2}}\right).
  \end{align}

\ \\

\textbf{Subcase (iii): $g = 3$:} As remarked above, there are no gaps of length 3 when $i$ is odd, and thus the contribution from $i$ even is the entire answer and we can immediately find that
\begin{align}
P_n(3)&\ =\ \cm\displaystyle\sum_{m=0}^{a_{2n+1}-1}\displaystyle\sum_{i=1\atop i\ {\rm even}}^{2n-3}X_{i,3}(m)\nonumber\\
&\ =\ \frac{1}{9}\cm \sum_{i =1\atop i\ {\rm even}}^{2n-3} (2^{\frac{i}{2}}+(-1)^{\frac{i-2}{2}}) \left(2^{\frac{2n-(i+4)}{2}+1}+(-1)^{\frac{2n-(i+4)}{2}}\right)\nonumber\\
&\ =\ \frac{1}{9}\cm \sum_{i =1\atop i\ {\rm even}}^{2n-3} 2^{n-1}+2^{\frac{i}{2}}(-1)^{\frac{2n-(i+4)}{2}}+2^{\frac{2n-(i+4)}{2}+1}(-1)^{\frac{i-2}{2}}+(-1)^{n-3}.
\end{align}
As the largest term in the above sum is $2^{n-1}$, we have
\be P_n(3) = \frac{c_n}{9}\left[(n-1)2^{n-1}+O(2^{n})\right].
\ee
Since $\mu_n \sim \frac{n}{3}$ and $a_{2n+1} \sim \frac{1}{3}(4\cdot2^n)$, using \eqref{eq:whatiscn} we find that up to lower order terms which vanish as $n\to\infty$ we have
\begin{align}\label{eq:whatiscmequalto}
\cm \ \sim\ \frac{9}{n2^{n+2}}.
\end{align}
Therefore
\begin{align}&P_n(3)\ \sim\  \frac{1}{n2^{n+2}}\cdot \left[(n-1)2^{n-1}+O(2^{n})\right]\ = \ \frac{1}{8}\cdot\frac{n-1}{n}+O\left(\frac{1}{n}\right).
\end{align}
Now as $n$ goes to infinity we see that $P(3)=1/8$.

\ \\

\noindent\textbf{Case 2: $i$ is odd:} Suppose now that $i$ is odd. The largest possible summand less than $a_i$ in a legal decomposition is $a_{i-3}$. As before we now need to know the number of integers that lie in $[0,a_{2\left(\frac{i-3}{2}\right)+1})$, but this equals \be a_{2\left(\frac{i-3}{2}\right)+1}\ =\ a_{2\left(\frac{i-1}{2}\right)-1}\ = \ \frac{1}{3}\left(2^{\frac{i-1}{2}+1}+(-1)^{\frac{i-1}{2}}\right).\ee

We now need to consider the parity of $i+g$.

\ \\

\textbf{Subcase (i): $g$ is odd:} When $i$ and $g$ are odd, we know $i+g$ is even and therefore the first possible summand greater than $a_{i+g}$ is $a_{i+g+3}$.  Like before, the number of  legal decompositions using summands from the set $\{a_{i+g+3},a_{i+g+4}, \dots, a_{2n}\}$ is the same as the number of $m$ with legal decompositions using summands from the set $\{a_1, a_2, \dots, a_{2n-(i+g+2)}\}$, which is $\frac{1}{3}\left(2^{\frac{2n-(i+g)}{2}+1}+(-1)^{\frac{2n-(i+g)}{2}}\right)$.  This leads to
\begin{align}
{\displaystyle\sum_{m=0}^{a_{2n+1}-1}}\displaystyle\sum_{i =1\atop i\ {\rm odd}, g\ {\rm odd}}^{2n-g}X_{i,g}(m)
 &\ = \ \frac{1}{9}\displaystyle\sum_{i =1\atop i\ {\rm odd}, g\ {\rm odd}}^{2n-g}(2^{\frac{i-1}{2}+1}+(-1)^{\frac{i-1}{2}})(2^{\frac{2n-(i+g)}{2}+1}+(-1)^{\frac{2n-(i+g)}{2}}).
  \end{align}

\ \\

\textbf{Subcase (ii): $g$ is even:} Following the same line of argument we see that if $i$ is odd and $g$ is even, then
\begin{align}
{\displaystyle\sum_{m=0}^{a_{2n+1}-1}}\displaystyle\sum_{i =1\atop i\ {\rm odd}, g\ {\rm even}}^{2n-g}X_{i,g}(m)
 &\ = \ \frac{1}{9}\displaystyle\sum_{i =1\atop i\ {\rm odd}, g\ {\rm even}}^{2n-g}(2^{\frac{i-1}{2}+1}+(-1)^{\frac{i-1}{2}})(2^{\frac{2n-(i+g+1)}{2}+1}+(-1)^{\frac{2n-(i+g+1)}{2}}).
  \end{align}

\ \\

Using these results, we can combine the various cases to determine the gap probabilities for different $g$. \\ \

\noindent\textbf{Part 2 of the Proof: Gap Probabilities:} \\ \

\noindent \textbf{Case 1: $g$ is  even:} As $g$ is even, we have $g=2j$ for some positive integer $j$. Therefore
\begin{align}
P_n(2j)&\ =\ \cm\displaystyle\sum_{m=0}^{a_{2n+1}-1}\displaystyle\sum_{i=1}^{2n-2j}X_{i,2j}(m)\\
&\ =\ \cm\displaystyle\sum_{m=0}^{a_{2n+1}-1}\displaystyle\sum_{i=1\atop i\ {\rm even}}^{2n-2j}X_{i,2j}(m)+\  \cm\displaystyle\sum_{m=0}^{a_{2n+1}-1}\displaystyle\sum_{i=1\atop i\ {\rm odd}}^{2n-2j}X_{i,2j}(m)\\
&\ = \cm\left[\frac{1}{9} \sum_{i =1 \atop i\ {\rm even}}^{2n-2j}(2^{\frac{i}{2}}+(-1)^{\frac{i-2}{2}})(2^{\frac{2n-(i+2j)}{2}+1}+(-1)^{\frac{2n-(i+2j)}{2}})\right]\\\nonumber &\;\;\;\;\;\;+\  \cm\left[ \frac{1}{9}\displaystyle\sum_{i =1\atop i\ {\rm odd}}^{2n-2j}(2^{\frac{i-1}{2}+1}+(-1)^{\frac{i-1}{2}})(2^{\frac{2n-(i+2j+1)}{2}+1}+(-1)^{\frac{2n-(i+2j+1)}{2}})\right]\\
&\ = \ \frac{1}{9}\cm\displaystyle\sum_{ i =1\atop i\ {\rm even}}^{2n-2j}(2^{n-j+1}+2^{\frac{i}{2}}(-1)^{\frac{2n-(i+2j)}{2}} + 2^{\frac{2n-(i+2j)}{2}+1}(-1)^{\frac{i-2}{2}}+(-1)^{n- j-1})\\\nonumber
 &\;\;\;\;\;\;\;\;+\frac{1}{9}\cm\displaystyle\sum_{i =1\atop i\ {\rm odd}}^{2n-2j}(2^{n-j+1}+2^{\frac{i-1}{2}+1}(-1)^{\frac{2n-(i+2j+1)}{2}}+ 2^{\frac{2n-(i+2j+1)}{2}+1}(-1)^{\frac{i-1}{2}}+(-1)^{n-j-1}).
  \end{align}

Notice that the largest terms in the above sums/expressions are given by $2^{n-j+1}$ and $ 2^{n-j+1}$, the sum of which gives $4(n-j)2^{n-j}$.
The rest of the terms are of lower order and are dominated as $n\to\infty$. Using \eqref{eq:whatiscmequalto} for $\cm$ we find
\begin{align}&P_n(2j)\ \sim\ \frac{\cm}{9} 4(n-j)2^{n-j}\ \sim\ \frac{1}{n2^{n+2}}\cdot 4(n-j)2^{n-j}\ = \ \frac{n-j}{n2^j},
\end{align}
and thus as $n$ goes to infinity we see that $P(2j)=1/2^j$.

\ \\

\noindent\textbf{Case 2: $g$ is odd:} As $g$ is odd we may write $g=2j+1$. Thus
\begin{align}
P_n(2j+1)&\ =\ \cm\displaystyle\sum_{m=0}^{a_{2n+1}-1}\displaystyle\sum_{i=1}^{2n-2j-1}X_{i,2j+1}(m)\\
&\ =\ \cm\displaystyle\sum_{m=0}^{a_{2n+1}-1}\displaystyle\sum_{i=1\atop i\ {\rm even}}^{2n-2j-1}X_{i,2j+1}(m)+\  \cm\displaystyle\sum_{m=0}^{a_{2n+1}-1}\displaystyle\sum_{i=1\atop i\ {\rm odd}}^{2n-2j-1}X_{i,2j+1}(m)\\
&\ = \cm\left[\frac{1}{9} \sum_{i =1\atop i\ {\rm even}}^{2n-2j-1} (2^{\frac{i}{2}}+(-1)^{\frac{i-2}{2}}) \left(2^{\frac{2n-(i+2j+2)}{2}+1}+(-1)^{\frac{2n-(i+2j+2)}{2}}\right)\right]\\\nonumber
&\;\;\;\;\;\;+\  \cm\left[\frac{1}{9}\displaystyle\sum_{i =1\atop i\ {\rm odd}}^{2n-2j-1}(2^{\frac{i-1}{2}+1}+(-1)^{\frac{i-1}{2}})(2^{\frac{2n-(i+2j+1)}{2}+1}+(-1)^{\frac{2n-(i+2j+1)}{2}}) \right]\\
&\ =\frac{1}{9}\cm \sum_{i =1\atop i\ {\rm even}}^{2n-2j-1} 2^{n-j}+2^{\frac{i}{2}}(-1)^{\frac{2n-(i+2j+2)}{2}}+2^{\frac{2n-(i+2j+2)}{2}+1}(-1)^{\frac{i-2}{2}}+(-1)^{n-j-2}\\\nonumber
&\;\;\;\;\;\;+\ \frac{1}{9}\cm\displaystyle\sum_{i =1\atop i\ {\rm odd}}^{2n-2j-1}2^{n-j+1}+2^{\frac{i-1}{2}+1}(-1)^{\frac{2n-(i+2j+1)}{2}}+2^{\frac{2n-(i+2j+1)}{2}+1}(-1)^{\frac{i-1}{2}}+(-1)^{n-j-1}.
  \end{align}
Notice that the largest terms in the above sums/expressions are given by $ 2^{n-j}$ and $2^{n-j+1}$, the sum of which gives $3(n-j)2^{n-j}$.
The rest of the terms are of lower order and are dominated as $n\to\infty$. Using \eqref{eq:whatiscmequalto} for $\cm$ we find
\begin{align}&P_n(2j+1)\ \sim\ \frac{\cm}{9} 3(n-j)2^{n-j}\ \sim\ \frac{1}{n2^{n+2}}\cdot3(n-j)2^{n-j}\ = \ \frac{3}{4}\cdot\frac{n-j}{n2^j},
\end{align} and thus as $n$ goes to infinity we see that $P(2j+1)=\frac{3}{4}\left(1/2^j\right)$.
\end{proof}

\section{Conclusion and Future Work}\label{sec:conclusionandfuturework}

Our results generalize Zeckendorf's theorem to an interesting new class of recurrence relations, specifically to a case where the first coefficient is zero. While we still have uniqueness of decomposition here, that is not always the case. In a future work \cite{CFHMN1} we study another example with first coefficient zero, the recurrence $a_{n+1} = a_{n-1} + a_{n-2}$. This leads to what we call the Fibonacci quilt, and there uniqueness of decomposition fails.

Additionally, the Kentucky-2  sequence is but one of infinitely many $(s,b)$-Generacci recurrences; in \cite{CFHMN2} we extend the results of this paper to arbitrary $(s,b)$.

\appendix


\section{Unique Decompositions}\label{sec:proofofK2unique}



\begin{proof}[Proof of Theorem \ref{K2unique}] Our proof is constructive. We build our sequence by only adjoining terms that ensure that we can {\em uniquely} decompose a number while never using more than one summand from the same bin or two summands from adjacent bins. The sequence begins:
\begin{equation} \underbracket{\ 1,\ 2\ }_{b_1}\ ,\ \underbracket{\    3,\  4 \ }_{b_2}\ ,\ \underbracket{\      5,\ 8 \ }_{b_3}\ ,\ \ldots.\end{equation}
Note we would not adjoin 9 because then 9 would legally decompose two ways, as $9=9$ and as $9=8+1$. The next number in the sequence must be the smallest integer that cannot be decomposed legally using the current terms.

We proceed with proof by induction. The basis case follows from a direct calculation. Notice that if $i\leq 5$ then $i=a_i$. Also
$6=a_5+a_1$.

The sequence continues:
\begin{equation} \ldots \ , \ \underbracket{\ a_{2n-5},\ a_{2n-4}\ }_{b_{n-2}}\ ,\ \underbracket{\    a_{2n-3},\ a_{2n-2} \ }_{b_{n-1}}\ ,\ \underbracket{\      a_{2n-1},\ a_{2n} \ }_{b_n}\ ,\ \underbracket{\     a_{2n+1},\ a_{2n+2} \ }_{b_{n+1}}\ ,\ \ldots\end{equation}
By induction we assume that there exists a unique decomposition for all integers $m\leq a_{2n}+w$, where $w$ is the maximum integer that legally can be decomposed using terms in the set $\{a_1, a_2, a_3$, $\dots$, $a_{2n-4}\}$. By construction we know that $w=a_{2n-3}-1$, as this was the reason we adjoined $a_{2n-3}$ to the sequence.

Now let $y$ be the maximum integer that can be legally decomposed using terms in the set $\{a_1, a_2, a_3$, $\dots$, $a_{2n}\}$. By construction we have
\begin{align}y\ = \ a_{2n}+w\ = \ a_{2n}+a_{2n-3}-1.\label{eq1}\end{align}
Similarly,  let $x$ be the maximum integer that legally can be decomposed using terms in the set  $\{a_1, a_2, a_3$, $\dots$, $a_{2n-2}\}$. Note $x = a_{2n-1}-1$ as this is why we include $a_{2n-1}$ in the sequence.

\bigskip
\noindent{\bf Claim:} $a_{2n+1}=y+1$ and this decomposition is unique.
\smallskip

By induction we know that $y$ was the largest value that we could legally make using only terms in $\{a_1, a_2, \dots, a_{2n}\}$. Hence we choose $y+1$ as $a_{2n+1}$ and $y+1$ has a unique decomposition.

\bigskip
\noindent{\bf Claim:} All $N \in {[y+1, y+1+x]=[a_{2n+1},a_{2n+1}+x]}$ have a unique decomposition.
\smallskip

We can legally and uniquely decompose all of $1,2,3,\ldots,x$ using elements in the set $\{a_1, a_2$, $\dots$, $a_{2n-2}\}$. Adding $a_{2n+1}$ to the  decomposition is still legal  since $a_{2n+1}$ is not a member of any bins adjacent to $\{b_1, b_2, \dots, b_{n-1}\}$. The uniqueness follows from the fact that if we do not include $a_{2n+1}$ as a summand, then the decomposition does not yield a number big enough to exceed $y+1$.


\bigskip
\noindent{\bf Claim:} $a_{2n+2}=y+1+x+1=a_{2n+1}+x+1$ and this decomposition is unique.
\smallskip

By construction the largest integer that legally  can be decomposed using terms  $\{a_1, a_2, \dots, a_{2n+1}\}$ is
$y + 1 + x$.

\bigskip
\noindent{\bf Claim:} All $N \in {[a_{2n+2}, a_{2n+2}+x]}$ have a unique decomposition.
\smallskip

First note that the decomposition exists as we can legally and uniquely construct $a_{2n+2}+v$, where $0\leq v\leq x$.
For uniqueness, we note that if we do not use $a_{2n+2}$, then the summation would be too  small. 

\bigskip
\noindent{\bf Claim:} $a_{2n+2}+x$ is the largest integer that legally  can be decomposed using terms  $\{a_1, a_2$, $\dots$, $a_{2n+2}\}$.
\smallskip

This follows from construction. \end{proof}





\section{Generating Function Proofs}\label{app:gfproof}

In \S\ref{sec:gaussianbehavior} we proved that the distribution of the number of summands in a Kentucky-2 decomposition exhibits Gaussian behavior by using properties of Fibonacci polynomials. This approach was possible because we had an explicit, tractable form for the $p_{n,k}$'s (Proposition \ref{Kentucky-Pnk-explicit}) that coincided with the explicit sum formulas associated with the Fibonacci polynomials.  Below we present a second proof of Gaussian behavior using a more general approach, which will be more useful in addressing the behavior of the number of summands when dealing with general $(s,b)$-Generacci sequences.
%

As in the first proof, we are interested in $g_n(y)$,  the coefficient of the $x^n$ term in $F(x,y)$.

\begin{lem} We have
\begin{eqnarray}\label{eq:gn2}
g_n(y) & \ = \  & \frac{1}{2^{n+1}\sqrt{1+8y}} \left[4y\left(1+\sqrt{1+8y}\right)^n-4y\left(1-\sqrt{1+8y}\right)^n\right. \nonumber\\ & & \left. \ \ \ \ \ \ \ \ \ + \ \left(1+\sqrt{1+8y}\right)^{n+1}-\left(1-\sqrt{1+8y}\right)^{n+1}\right].
\end{eqnarray}
\end{lem}

\begin{proof} For brevity set $x_1=x_1(y)$ and $x_2=x_2(y)$ for the roots of $x$ in $x^2+\frac{1}{2y}x-\frac{1}{2y}$.
In particular, we find
\begin{align}\label{xs} x_1 \ = \  -\frac{1}{4y}\left(1+\sqrt{1+8y}\right) \ \ \ \
x_2 \ = \  -\frac{1}{4y}\left(1-\sqrt{1+8y}\right).
\end{align}
Since $x_1$ and $x_2$ are unequal for all $y>0$, we can decompose $F(x,y)$ using partial fractions:
\begin{equation}F(x,y) \ = \  \frac{1+2xy}{-2y(x-x_1)(x-x_2)} \ = \   \frac{1+2xy}{-2y}\frac{1}{x_1-x_2}\left[  \frac{1}{x-x_1}-\frac{1}{x-x_2}  \right].\end{equation}
Using the geometric series formula, after some algebra we obtain
\be\label{gny} F(x,y) \ = \  \frac{1+2xy}{-2y}\frac{1}{x_1-x_2}\sum_{i\geq0}\left[  \frac{1}{x_1}\left(\frac{x}{x_1}\right)^i-\frac{1}{x_2}\left(\frac{x}{x_2}\right)^i \right] .  \ee
From here we find that that the coefficient of $x^n$ is
\begin{equation} g_n(y) \ = \  \frac{1}{-2y(x_1-x_2)}\left[\frac{1}{x_1^{n+1}}-\frac{1}{x_2^{n+1}}+\frac{2y}{x_1^{n}}-\frac{2y}{x_2^{n}}  \right].  \end{equation}
Substituting the functions from Equation \eqref{xs} and simplifying we obtain the desired result.
\end{proof}

As we mentioned in \S\ref{sec:mv}, we have the following corollary.

\begin{cor} Let $F_n(x)$ be a Fibonacci polynomial. Then
\begin{equation}
F_n(x) \ = \  \frac{(x+\sqrt{x^2+4})^n - (x-\sqrt{x^2+4})^n}{2^n \sqrt{x^2+4}}.\end{equation}
\end{cor}

\begin{proof}
Set the righthand sides of Equations \eqref{eq:gnFib} and \eqref{eq:gn2} equal and let $x=1/\sqrt{2y}$.
\end{proof}

\begin{proof}[Proof of Proposition \ref{prop:mean}]
Straightforward, but somewhat tedious, calculations give \bea g_n(1) &\ =\ & \frac{1}{3}\left((-1)^{n+1}+2^{n+2}\right) \nonumber\\
g'_n(1) & \ = \ & \frac{n}{9}\left(2^{n+2}+2(-1)^{n+1}\right)+\frac{2}{27}\left(2^{n+2}\right)+o(1).
\eea
Dividing these two quantities and using Lemma \ref{lem:mu} gives the desired result.
\end{proof}

\begin{proof}[Proof of Proposition \ref{prop:var}]
Another straightforward (and again somewhat tedious) calculation yields
\begin{align}
\sigma_n^2&\ = \ \frac{2^{2n+5}(4+3n)-2(8+3n)-2^{n+2}(-1)^{n}(28+36n+9n^2)}{81(2^{n+2}-(-1)^{n})^2}\nonumber\\
&\ = \ \frac{n\big[(6)2^{2n+4}-18(-1)^n 2^{n+3}-6\big]+\big[ (8) 2^{2n+4}-14(-1)^n 2^{n+3}-16  \big]-4.5(-1)^n n^2 2^{n+3}}{81\big[2^{2n+4}-(-1)^n 2^{n+3}+1\big]}.
\end{align}
\end{proof}

\begin{proof}[Proof of Theorem \ref{thm:gaussian}]
As in our earlier proof, we show that the moment generating function of $Y_n'$ converges to that of the standard normal. Following the same argument as in \cite[Lemma 4.9]{DDKMMV}, the moment generating function $M_{Y_n'}(t)$ of $Y_n'$ is
\begin{equation}
M_{Y_n'}(t) \ = \  \frac{g_n(e^{t/\sigma_n})e^{-t\mu_n/\sigma_n}}{g_n(1)}.
\end{equation}
Taking logarithms yields
\begin{equation}\label{logM} \log M_{Y_n'}(t) \ = \  \log [g_n(e^{t/\sigma_n})]-\log[g_n(1)]-\frac{t\mu_n}{\sigma_n}.  \end{equation}
We tackle the right hand side in pieces.

Let $r_n=t/\sigma_n$. Since $\sigma_n^2 = \frac{2n}{27}+\frac{8}{81}+O\left(\frac{n^2}{2^n}\right)$,  as $n$ goes to infinity $r_n$ goes to 0. This allows us to use Taylor series expansions.

First we rewrite $g_n(e^{r_n})$
\begin{align}
 g_n(e^{r_n}) =& \frac{1}{\sqe}\left[  \frac{(1+\sqe)^n(4e^{r_n}+1+\sqe)}{2^{n+1}} \right. \nonumber\\
&\qquad\qquad\left. -\frac{4e^{r_n}(1-\sqe)^n}{2^{n+1}}-\frac{(1-\sqe)^{n+1}}{2^{n+1}}     \right]  .   \end{align}
Using Taylor series expansions of the exponential and square root functions we obtain \begin{equation}
e^{r_n}= 1+ o(1)\quad \mbox{and} \quad \frac{1-\sqe}{2} = -1 +o(1).
\end{equation} Thus
\begin{align}
  \frac{4e^{r_n}(1-\sqe)^n}{2^{n+1}}+\frac{(1-\sqe)^{n+1}}{2^{n+1}} &\ = \ 2(-1)^n +o(1) -(-1)^n+o(1) \nonumber\\
&\  = \   (-1)^n+o(1) . \end{align}
Hence
\begin{equation}   g_n(e^{r_n}) = \frac{1}{\sqe}\left[  \frac{(1+\sqe)^n(4e^{r_n}+1+\sqe)}{2^{n+1}}  -   (-1)^n+o(1) \right]  .    \end{equation}
So
\begin{align}
\log(g_n(e^{r_n}))\ = \ &-\tfrac{1}{2}\log(1+8e^{r_n})+n\log(1+\sqe)\nonumber\\
&+\log(4e^{r_n}+1+\sqe)-(n+1)\log2+o(1).
\end{align}
Continuing to use Taylor series expansions
\begin{align}\nonumber
\log(g_n(e^{r_n}))=&-\tfrac{1}{2}\left[\log9 +\frac{8}{9}r_n+\frac{4}{81}r_n^2\right]+n\left[\log4 +\frac{1}{3}r_n+\frac{1}{27}r_n^2\right]\\\label{dn2}
&\ \ \ \ +\ \left[\log8 +\frac{2}{3}r_n+\frac{2}{27}r^2_n\right]+O(r_n^3)-(n+1)\log2+o(1).
\end{align}

Finally, recall $g_n(1) = \frac{1}{3}[(-1)^{n+1}+2^{n+2}]$ so
\be\label{dn3}\log[g_n(1)] \ = \ -\log 3+(n+2)\log2 +o(1).\ee

To finish we plug values into Equation \eqref{logM}. In particular, plug in $\log(g_n(e^{r_n}))$ from Equation \eqref{dn2}, $\log[g_n(1)]$ from Equation \eqref{dn3}, $\mu_n$ from Proposition \ref{prop:mean}, $\sigma_n$ from Proposition \ref{prop:var}, and $r_n=t/\sigma_n$. This gives
\begin{equation} \log M_{Y_n'}(t) = \frac{t^2}{2}+o(1).  \end{equation}
Thus, $M_{Y_n'}(t)$ converges to the moment generating function of the standard normal distribution. Which according to probability theory, implies that
the distribution of $Y_n'$ converges to the standard normal distribution.
\end{proof}


\ \\

\end{document}